\documentclass[12pt]{amsart}

\usepackage{amssymb,latexsym,amsmath,extarrows, mathrsfs, amsthm}
\usepackage{graphicx}

\usepackage[margin=1in, centering]{geometry}

\usepackage{color}
\usepackage{wasysym}

\newtheorem{theorem}{Theorem}[section]
\newtheorem{lemma}[theorem]{Lemma}
\newtheorem{proposition}[theorem]{Proposition}
\newtheorem{remark}[theorem]{Remark}

\newtheorem{corollary}[theorem]{Corollary}

\newtheorem{conjecture}[theorem]{Conjecture}

\newcommand{\R}{\mathbb R}

\newcommand{\ZZ}{\mathbb{Z}}

\newcommand{\ZQ}{\mathbb{Q}}

\newcommand{\tM}

\allowdisplaybreaks

\makeatletter
\@namedef{subjclassname@2020}{\textup{2020} Mathematics Subject Classification}
\makeatother

\title{Asymptotic estimate on the distance energy of lattices}
\author{Zhipeng Lu}
\address{Shenzhen MSU-BIT University, 1 International University Park Road, Dayun New Town, Longgang District, Shenzhen, Guangdong Province, P.R. China}

\email{zhipeng.lu@hotmail.com}
\author{Xianchang Meng}
\address{School of Mathematics, Shandong University, Jinan, Shandong 250100, China \& Johannes Kepler University Linz
Altenberger Stra\ss e 69, Linz 4040, Austria}

\email{xianchang.meng@gmail.com}

\date{}

\begin{document}
\begin{abstract}
	Since the well-known breakthrough of L. Guth and N. Katz on the Erd\H{o}s distinct distances problem in the plane, mainstream of interest is aroused by their method and the Elekes-Sharir framework. In short words, they study the second moment in the framework. One may wonder if higher moments would be more efficient. In this paper, we show that any higher moment fails the expectation. In addition, we show that the second moment gives optimal estimate in higher dimensions.
\end{abstract}
\maketitle

\section{Introduction}
The Erd\H{o}s conjecture on distinct distances in the Euclidean plane $\R^2$ says, $d(P):=|\{d(p,q)\mid p,q\in P\}|\gtrsim \frac{|P|}{\sqrt{\log|P|}}$ for any finite set $P\subset\R^2$. Here ``$\gtrsim$" means ``$\geq C\cdot$" for some absolute constant $C>0$ and $d(,)$ is the Euclidean distance. In Guth-Katz \cite{GK}, the nearly optimal bound $d(P)\gtrsim\frac{|P|}{\log|P|}$ was established. The authors used a group theoretic framework, called the Elekes-Sharir framework, to reduce the problem of enumerating distinct distances to that of estimating line-line incidences in $\R^3$. 

The essential object they studied is a type of energy, which they call \textit{distance quadruples}, i.e. $Q(P)=:\{(p_1,q_1,p_2,p_2)\in P^4\mid d(p_1,q_1)=d(p_2,q_2)\}$. We call $|Q(P)|$ the \textbf{distance energy} of $P$, denoted by $E_{2}(P)$. Moreover, we can define $E_{k}(P):=|\{(p_1,q_1,\dots,p_k,q_k)\in P^{2k}\mid d(p_1,q_1)=\cdots=d(p_k,q_k)\}|$, and call it the $k$-th distance energy of $P$.

In this paper, we consider higher distance energy $E_k(P)$ for $k\geq 3$ and investigate higher moments in the Elekes-Sharir framework with the expectation that it might be more efficient than just estimating $E_2(P)$ as in \cite{GK}. Due to technical reasons, we only consider the example of square grids, which yet shows that the expectation is in vain for the Euclidean plane. Moreover, we also study distance energies in $\R^n, n\geq3$ for the square grid example.
\\

{\bf Acknowledgements.} ZL is supported by Guangdong Basic and Applied Basic Research Foundation (No. 2021A1515110206). XM is supported by the National Natural Science Foundation of China (NSFC, Grant No. 12201346) and Shandong Provincial Foundation (Grant No. 2022HWYQ-046).

\section{Higher moments in the Elekes-Sharir framework}
We may define distance energy in any general metric space $(M,d)$. Let $P\subset M$ be a set of $N$ points and $d(P)$ the number of distinct distances between points of $P$. For each distance $d_i,i=1,\cdots,d(P)$, let $n_i$ be the number of pairs of points from $P$ at distance $d_i$. Clearly $\sum_{i=1}^{d(P)}n_i=2\binom{N}{2}=N(N-1)$. Then we use H\"{o}lder's inequality to get the following
\begin{lemma}\label{lem-Holder}For any positive integer $k\geq 2$,
\[d(P)\geq\dfrac{(N^2-N)^{\frac{k}{k-1}}}{(\sum_{i=1}^{d(P)}n_i^k)^{\frac{1}{k-1}}}.\]
\end{lemma}
\begin{proof}
By H\"{o}lder's inequality for $\frac{k}{k-1}+\frac{1}{k-1}=1$,
\[N^2-N=\sum_{i=1}^{d(P)}n_i\leq \left(\sum_{i=1}^{d(P)}1^{\frac{k}{k-1}}\right)^{\frac{k-1}{k}}\left(\sum_{i=1}^{d(P)}n_i^k\right)^{1/k}={d(P)}^{\frac{k-1}{k}}\left(\sum_{i=1}^{d(P)}n_i^k\right)^{1/k}.\]
Rearrange we get the desired inequality.
\end{proof}
By our definition, $E_k(P)=\sum_{i=1}^{d(P)}n_i^k$. In order to prove Erd\H{o}s' conjecture in this setting, we need to show
\begin{equation}\label{eq-conjecture}
E_k(P)\lesssim N^{k+1}\left(\log N\right)^{\frac{k-1}{2}}, \forall P\subset\R^2 \text{ with } |P|=N,
\end{equation}
at least for some $k\geq2$.
Guth and Katz \cite{GK} already showed that this is not true for $k=2$. Actually, they proved $E_2(P)\lesssim N^3\log N$. They also calculated for the example where $P$ is a square grid with $N$ points, $E_2(P)\gtrsim N^3\log N$, see appendix of \cite{GK}. Thus, to verify or refute (\ref{eq-conjecture}), we need to estimate $E_k(P)$ at least for the square grid example.

\section{Rough estimate on higher distance energy of square grids}
In this section, we use number theoretic method, specifically, by calculating Dirichlet series on the number of expression as sum of squares derived from higher distance energies $E_k(P)$. Note that in the appendix of \cite{GK}, $E_2(P)$ was estimated by counting line-line incidences in $\R^3$.

Let $P=[\sqrt{N}]\times[\sqrt{N}]$ be the square grid of size $N$, where $[x]$ denotes the set of integers ranging from $1$ to $\lfloor x\rfloor$. Then a piece of energy in $E_k(P)$ is afforded by $a_1^2+b_2^2=\cdots=a_k^2+b_k^2$ for some $a_i,b_i\in[\sqrt{N}], i=1,\dots,k$, and some $p_i,q_i\in P$ such that $p_i-q_i=(\pm a_i,\pm b_i)$. Note that for $a_i,b_i\leq\frac{\sqrt{N}}{2}$, the number of such pairs $(p_i,q_i)$ is $\gtrsim \sqrt{N}\cdot\sqrt{N}=N$. Denoting $r(n):=|\{(a,b)\in\ZZ^2\mid a^2+b^2=n\}|$, we get the rough estimate written as
\begin{align}\label{eq-energy estimate}
N^{k}\sum_{n\leq N}r(n)^k\gtrsim E_k(P)\gtrsim N^k\sum_{n\leq\frac{N}{2}}r(n)^k.
\end{align}
%For convenience, we let $r(0)=2$ rather than $1$ to ease analytic calculation in later sections, since for $n\neq0$, 2 points $p,q$ affording $r(n)$ are counted as two pairs. 
On average, the number of square sum expressions have the following estimate
\begin{proposition}\label{prop-sq sum est.}
	For any positive integer $k$ and $x\in\R_+$, we have the following asymptotics 
	\[\sum_{n\leq x}r(n)^k\sim xP_{2^{k-1}-1}(\log x),\]
	where $P_{2^{k-1}-1}$ is a polynomial of degree $2^{k-1}-1$.
\end{proposition}
Note that more precisely for $k=2$, \begin{equation}\label{eq-k=2}
\sum_{n\leq x}r^2(n)\sim 4x\log x+O(x)
\end{equation}
(see Wilson \cite{Wilson}).
The lemma is a detour of consequence by the following two results.
\begin{lemma}[(7.20) of \cite{Wilson}]\label{lem-Wilson}
	For any positive integer $k$, there is the following expression of Dirichlet series:
	\[\sum_{n=1}\frac{r(n)^k}{n^s}=4^k(1-2^{-s})^{2^{k-1}-1}\left(\zeta(s)\eta(s)\right)^{2^{k-1}}\phi(s), \forall\Re(s)>1,\]
	where $\zeta(s)$ is the Riemann zeta function, $\eta(s)=1^{-s}-3^{-s}+5^{-s}-7^{-s}+\cdots$, and $\phi(s)=\prod_{p}\left(1+\sum_{\nu=2}^\infty a_\nu p^{-\nu s}\right)$ is absolutely convergent for $\Re(s)>\frac{1}{2}$.
\end{lemma}
To get the average of $r^k(n)$, we rely on the follow form of Perron's formula:
\begin{lemma}[Theorem 1 of Chapter V in Karatsuba \cite{Karatsuba}]\label{lem-Perron}
	Assume that the Dirichlet series $f(s)=\sum\limits_{n=1}^\infty\frac{a_n}{n^s}$ converges absolutely for $\Re(s)>1$, $|a_n|\leq A(n)$ for some monotonically increasing function $A(x)>0$, and 
	\[\sum_{n=1}^\infty\frac{|a_n|}{n^\sigma}=O((\sigma-1)^{-\alpha}), \alpha>0,\]
	as $\sigma\rightarrow 1_+$. Then for any $b_0\geq b>1$, $T\geq 1$, and $x=N+\frac{1}{2}$, we have 
	\[\sum_{n\leq x}a_n=\frac{1}{2\pi i}\int_{b-iT}^{b+iT}f(s)\frac{x^s}{s}ds+O\left(\frac{x^b}{T(b-1)^{\alpha}}\right)+O\left(\frac{xA(2x)\log x}{T}\right).\]
\end{lemma}
\iffalse
\begin{lemma}[Theorem of \cite{GKLN}]\label{lem-GKLN}
	Assume that $0\leq a(n)=o(n^\epsilon)$ for every $\epsilon>0$, and its associated Dirichlet series can be expressed as
	\[\sum_{n=1}^\infty\frac{a(n)}{n^s}=\frac{\prod_{m=1}^M\zeta_{K^*_m}(s)}{\prod_{j=0}^J\left(\zeta_{K_j}(s)\right)^{\tau_j}}G(s), \forall\Re(s)>1.\]
	Here $\zeta_{K^*_m}(s)$ are Dedekind zeta functions on some number fields $K^*_m$ of degree $d_m\in\{1, 2\}$ over $\ZQ$, $d_1+\cdots+d_M=4$, $K_j$ are arbitrary algebraic number fields, $\tau_j$ are fixed real numbers, $J\geq0$, and $G(s)$ is a ``harmless factor" (holomorphic, bounded away from 0 uniformly for $\Re(s)\geq\sigma$, $\sigma<\frac{1}{2}$). Then
	\[\sum_{n\leq x}a(n)\sim x P_{M-1}(\log x),\]
	for some degree $M-1$ polynomial $P_{M-1}$.
\end{lemma}
\fi 

\begin{proof}[Proof of Proposition \ref{prop-sq sum est.}]
	First, $r(n)$ is indeed always of order $o(n^\epsilon)$ for any $\epsilon>0$ (see for instance, Theorem 338 of Hardy and Wright \cite{HW}). Thus the condition of Lemma \ref{lem-Perron} is easily satisfied. Also note that $\eta(s)$ is holomorphic and poles are on only $\zeta(s)$. Then Wilson's calculation of the Dirichlet series of $r(n)^k$ as in Lemma \ref{lem-Wilson} shows that, by estimating the residue integral of contour, the order of $\sum_{n\leq x}r(n)^k$ should be $x(\log x)^{2^{k-1}-1}$, whereas $T$ may tend to be larger than any log power due to irrelevance of choices of contours. 
	
\end{proof}
Proposition \ref{prop-sq sum est.} together with (\ref{eq-energy estimate}) immediately implies the following
\begin{theorem}
	For $k\geq2$ and large positive integer $N$ ($\gg k$), we have the following estimate on the $k$-th distance energy:
	\[N^{k+1}(\log N)^{2^{k-1}-1}\gtrsim E_k([\sqrt{N}]\times[\sqrt{N}])\gtrsim_k N^{k+1}(\log N)^{2^{k-1}-1}.\]
\end{theorem}
The result turns against the expectation of (\ref{eq-conjecture}) by a big log factor. Moreover, if $N^{k+1}(\log N)^{2^{k-1}}$ is the right order for $E_k(P)$ in general, then by Lemma \ref{lem-Holder}, in the Elekes-Sharir framework we may get the most efficient bound $d(P)\gtrsim\frac{N}{\log N}$ only when we study the second moment. 
\section{Distance energy of square grids in higher dimensions}
In addition, we notice that the distance energy of square lattices in higher dimensions is optimal. Consider $P=[\sqrt[m]{N}]^m$ the square grid of size $N$ in $\R^m, m\geq 3$, and let $r_m(n)=|\{(x_1,\dots,x_m)\in\ZZ^m\mid x_1^2+\cdots+x_m^2=n\}|$. Then similarly we have 
\[N^2\sum\limits_{n\leq \frac{N^{2/3}}{2}}r_3^2(n)\leq E_2(P)\leq N^2\sum\limits_{n\leq N^{2/3}}r_3^2(n).\]
To this end, we introduce a more general result as follows
\begin{lemma}[Theorem 6.1 of M\"{u}ller \cite{Muller}]\label{lem-Muller}
	Let $q(\mathbf{x})=\frac{1}{2}\mathbf{x}^TQ\mathbf{x}$ be a primitive positive definite integral quadratic form in $m\geq 3$ variables and $r_Q(n)=|\{\mathbf{x}\in\ZZ^m\mid q(\mathbf{x})=n\}|$. Then
	\[\sum_{n\leq x}r_Q^2(n)=Bx^{m-1}+O\left(x^{(m-1)\frac{4m-5}{4m-3}}\right),\]
	for some constant $B>0$ depending on $Q$. For $m=2$, 
	\[\sum_{n\leq x}r_Q^2(n)=A_Qx\log x+O(x),\]
	where \[A_Q=12\frac{A(q)}{q}\prod_{p\mid q}\left(1+\frac{1}{p}\right)^{-1},\ q=\det(Q).\]
	Here $A(q)$ denotes the multiplicative function defined by $A(p^e)=2+(1-\frac{1}{p})(e-1)$ for odd prime $p$, and 
	\[A(2^e)=\begin{cases}
	1, \text{ if }e\leq 1,\\
	2  \text{ if } e=2,\\
	e-1, \text{ if }e\geq 3.
	\end{cases}\]
\end{lemma}
This immediately implies
\begin{corollary}\label{cor-3 d}
	For $P=[\sqrt[m]{N}]^m$ the square grid of size $N$ in $\R^m, m\geq 3$, we have
	\[N^{2+\frac{2m-2}{m}}\lesssim_m E_2(P)\lesssim_m N^{2+\frac{2m-2}{m}}.\]
\end{corollary}
Note that by Legendre's three-square theorem, $n=x^2+y^2+z^2\leq N^{\frac{2}{3}}$ for $n\neq 4^a(8m+7)$, which amount to $cN^{\frac{2}{3}}$ numbers for some $c>0$, i.e. $d(P)=c'N^{\frac{2}{3}}, c'>0$, as Erd\H{o}s noted for the distinct distances conjecture in $\R^3$. For $m\geq 4$, by Lagrange's four-square theorem, each positive integer can be expressed as a sum of $m$ squares, i.e. $d(P)=c{N}^{\frac{2}{m}}$. Thus for any $m\geq3$, we can conclude that
\begin{corollary}
	For any $m\geq3$, the estimate by distance energy of distinct distances in $\R^m$ is optimal, i.e.
	\[d(P)\lesssim\frac{|P|^4}{E_2(P)},\]
	for certain examples like $P=[\sqrt[m]{N}]^m$ the square grid of size $N$.
\end{corollary}
This seems to indicate the following
\begin{conjecture}
   $E_{2}(P)\lesssim |P|^{2+\frac{2m-2}{m}}$ for any finite set $P\subset\R^m$. 
\end{conjecture}
A proof of the above estimate for distance energy suffices to solve the Erd\H{o}s conjecture in higher dimensions. However, similar as in $\R^2$, we believe that estimate by higher distance energies would not be optimal.

\section{The distance energy for general lattices and Epstein zeta functions}
In this section, we consider general lattices in $\R^2$ and compare their distance energy. Due to technical reasons, we only deal with the pointwise distance energy, i.e. $E_{L,k}(N):=|\{(p_1,\dots,p_k)\in L^k, \|p_1\|^2=\cdots=\|p_k\|^2\leq N\}|$ for any lattice $L\subset\R^2, k\in\ZZ_{\geq 0}$. Let $r_L(n)=|\{p\in L, \|p\|^2=n\}|$. Then 
\begin{equation}
E_{L,k}(N)=\sum_{n\leq N}r_L(n)^k.
\end{equation}
We have already seen the estimates of $E_{L,k}(N)$ for the square lattice $L$ in the last section. Note that $E_{L,0}(N)$ counts the distinct distances.

A general lattice $L\subset\R^2$ of covolume 1, after rotation, may be written as $\ZZ (a,0)\oplus\ZZ(b,\frac{1}{a})$ for some $a,b>0$. To estimate its distance energy, we need to study the value distribution of the quadratic form $Q_L(x,y)=(ax+by)^2+\frac{1}{a^2}y^2=a^2\left(x^2+2\frac{b}{a}xy+\left(\frac{1}{a^4}+\frac{b^2}{a^2}\right)y^2\right)$. For example, $a=\sqrt{\frac{2}{\sqrt{3}}},b=\frac{1}{2}\sqrt{\frac{2}{\sqrt{3}}}$ correspond to the hexagonal lattice (which we will always denote by $\Sigma$) and the quadratic form $\frac{2}{\sqrt{3}}(x^2+xy+y^2)$. If $\frac{a}{b}$ or $\frac{1}{a^4}+\frac{b^2}{a^2}$ is irrational, then integer solutions to $Q_L(x,y)=Q_L(x',y')$ would be very few, i.e. have small distance energy. 

We will only concern about the lattices with $Q_L$ similar to norms of imaginary quadratic number fields, i.e. \textit{arithmetic} lattices, due to the following K\"{u}hnlein's criterion:
\begin{lemma}\label{lem-Kuhnlein}[K\"{u}hnlein \cite{Kuhn}]
	Let $L\subset\R^2$ be a lattice. Then $L$ is arithmetic if and only if there are $\geq 3$ pairwise linearly independent vectors in $L$ which have the same length.	
\end{lemma}
This immediately implies
\begin{corollary}\label{cor-arithmetic}
Any distance in a non-arithmetic lattice is at most repeated 4 times. Hence the number of distinct distances in a non-arithmetic lattices of $N$ points is $\gtrsim N$ and its distance energy is $O(N^2)$.	
\end{corollary}
\begin{remark}
	For $L$ arithmetic, there is always $E_{L,0}(N)\lesssim\frac{N}{\sqrt{\log N}}$ like the square grids, see Moree and Osburn \cite{MO} for details. Moreover, they proved that in $\R^2$ the hexagonal lattice attains the minimal number of distinct distances, i.e. the minimum of $E_{L,0}(N)$ for $N$ large. 
\end{remark}

Now that arithmetic lattices are nothing but submodules of rings of integers of imaginary quadratic fields, it suffices to consider the norms of those rings. For any negative square-free integer $D$, if $D\equiv1\mod 4$, the ring of integers is $\mathcal{O}_D=\ZZ\left[\frac{1+\sqrt{D}}{2}\right]$ with \textit{discriminant} $D$; otherwise, $\mathcal{O}_D=\ZZ[\sqrt{D}]$ with discriminant $4D$. Hence we define
\begin{equation}\label{eq-Q_D}
 Q_{D}(x,y)=\begin{cases}
 (x+\frac{1+\sqrt{D}}{2}y)(x+\frac{1-\sqrt{D}}{2}y)=x^2+xy+\frac{1-D}{4}y^2,\ D\equiv1\mod 4;\\
 (x+\sqrt{D}y)(x-\sqrt{D}y)=x^2-Dy^2,\ \text{otherwise}.
\end{cases}
\end{equation}
Note that the discriminant is just that of $Q_{D}$ and $Q_D$ are all positive definite. For example, if $D=-1$, it is the Gaussian ring $\ZZ[i]$ with norm $Q_{-1}(x,y)=x^2+y^2$; if $D=-3$, it is the Eisenstein ring $\ZZ\left[\frac{1+\sqrt{3}i}{2}\right]$ with norm $Q_{-3}(x,y)=x^2+xy+y^2$. To have covolume 1, the lattices need to be scaled by $S_D:=\sqrt{2}(-D)^{-\frac{1}{4}}$ or $(-D)^{-\frac{1}{4}}$, which was already seen in the case of hexagonal lattice. This is to assure the justness that there are always $\sim \pi N$ lattice points in a disc of radius $\sqrt{N}$. 

For arithmetic lattices, we may write $r_D(n)$ for $r_L(n)$. Note that $r_{-1}(n)=r(n)$ and $r_{-3}(n)=|\{(x,y)\in\ZZ^2\mid x^2+xy+y^2=n\}|$ are the only two cases of counting integral points on circles, while the others on ellipses. Then we also write \begin{equation}E_{D,k}(N)=\sum_{n\leq N/S_D^2}r_{D}^k(n).\end{equation}
Then we can use Lemma \ref{lem-Muller} to give the asymptotics of $E_{D,2}(N)$. By calculation we see that
\begin{align}
E_{-3,2}(N)&=3\sqrt{3}N\log N+O(N),
\end{align}
which is larger than $E_{-1,2}(N)=4N\log N+O(N)$ as we have seen from (\ref{eq-k=2}). Note that the $2\times 2$ matrices as of Lemma \ref{lem-Muller} are 
\[\begin{cases}
\begin{pmatrix}
2&1\\
1\frac{1-D}{2}
\end{pmatrix}, \text{ if }D\equiv 1\mod{4},\\
{}\\
\begin{pmatrix}
2&0\\
0&-2D
\end{pmatrix}, \text{ otherwise}.
\end{cases}\]
By more careful calculation on coefficients of the main terms, we see the following
\begin{theorem}\label{thm-energy for lattices}
	Let $D$ be any square free negative integer and $N$ large. Then $E_{D,2}(N)<E_{-3,2}(N)$ for $D\equiv 1\mod{4}$, and $E_{D,2}(N)<E_{-1,2}(N)$ otherwise. In all, the pointwise distance energy $E_{L,k}(N)$ attains the maximum only when $L$ is the hexagonal lattice in $\R^2$.
\end{theorem}

One may also be interested in higher pointwise energies $E_{D,k}(N)$ for $k\geq 3$. Explicit formulas of $r_D(n)$ may be found in Huard, Kaplan and Williams \cite{HKW} or Sun and Williams \cite{SW}, but estimating $E_{D,k}(N)$ from those formulae is hardly possible.

On the other hand, in general $r_Q(n)=|\{(x,y)\in\ZZ^2\mid Q(x,y)=n\}|$ is used to define the Epstein zeta function:
\begin{equation}\label{eq-Epstein}
Z_{Q}(s)=\sum_{m,n\neq 0}\frac{1}{Q(m,n)^s}=\sum_{n=1}^\infty\frac{r_Q(n)}{n^s}.
\end{equation}
which converges for $\Re{s}>1$. Moreover, it can be analytically continued to the whole complex plane with a simple pole at $s=1$ and satisfies the functional equation ($D=disc(Q)$)
\begin{equation}\label{eq-Epstein functional}
\left(\frac{\sqrt{D}}{2\pi}\right)^s\Gamma(s)Z_Q(s)=\left(\frac{\sqrt{D}}{2\pi}\right)^{1-s}\Gamma(1-s)Z_Q(1-s),
\end{equation}
see for instance Zhang and Williams \cite{ZW}. There is also a closed formula by Chowla and Selberg (see \cite{CS}), which states for $Q(x,y)=ax^2+bxy+cy^2, D=b^2-4ac$
\begin{align}\label{eq-Selberg Chowla formula}
Z_Q(s)&=a^{-s}\zeta(2s)+a^{-s}\sqrt{\pi}\frac{\Gamma(s-\frac{1}{2})}{\Gamma(s)}\zeta(2s-1)l^{1-2s}+R_Q(s),\\
R_Q(s)&=\frac{4a^{-s}l^{-s+\frac{1}{2}}}{\pi^{-s}\Gamma(s)}\sum_{n=1}^\infty n^{s-\frac{1}{2}}(\sum_{d\mid n}d^{1-2s})K_{s-\frac{1}{2}}(2\pi nl)\cos(\frac{n\pi b}{a}),\notag
\end{align} 
where $K_\nu(z)$ is a modified Bessel function, $l=\frac{\sqrt{|D|}}{2a}$.

To investigate distribution of the higher distance energies $E_{D,k}(N)$, we initiate the study of higher moments of the Epstein zeta functions, i.e.
\begin{equation}\label{eq-Epstein higher moments}
Z_{Q,k}(s):=\sum_{n=1}^\infty\frac{r_Q(n)^k}{n^s},\ k\geq 3.
\end{equation}

\textit{Question}: Do these higher moments satisfy any functional equation or have closed formulae as of (\ref{eq-Epstein functional}) or (\ref{eq-Selberg Chowla formula})?\\

If they do, then we should be able to derive asymptotics for the average of $r_D^k(n)$ by Perron's formula as of Lemma \ref{lem-Perron}. It has been shown that $Z_{Q}(s)$ attains the minimum only for equivalent forms of $Q_{-3}$, i.e. for the hexagonal lattice, whenever $s\geq 0$, see Cassels \cite{Ca}. Thus, we wonder if this is true for all the higher moments and suggest the following

\begin{conjecture}
	For all $k\geq 1$, $Z_{Q,k}(s)\geq Z_{Q_{-3},k}(s), \forall s>1$. After analytic continuation (if there is), this should be true for all $s\geq0$.
\end{conjecture} 


\begin{thebibliography}{9}
\bibitem{Ca}
J. W. S. Cassels, On a problem of Rankin about the Epstein zeta function, \textit{Proc. Glasg. Math. Assoc.} \textbf{4}, 73-80 (1959); \textbf{6}, 116 (1963).
\bibitem{CKO}
S. K. K. Choi, A. V. Kumchev, R. Osburn, On sums of three squares, \textit{International Journal of Number Theory}, Vol. \textbf{1}, No. 2 (2005), 161-173.
\bibitem{ES}
G. Elekes, M. Sharir, {Incidences in three dimensions and distinct distances in the pane}, \textit{Combin. Probab. Comput.} \textbf{20} (2011), 571-608.
%\bibitem{Guth}
\bibitem{GK}
L. Guth, N. H. Katz, On the Erd\H{o}s distinct distances problem in the plane, \textit{Annals of Mathematics} \textbf{181} (2015), 155-190.
%\bibitem{GKLN}
%M. Z. Garaev, M. K\"{u}hleitner, F. Luca, W. G. Nowak, Asymptotic formulas for certain arithmetic functions, \textit{Math. Slovaca} \textbf{58} (2008), No. 3, 301-308.
%\bibitem{Handbook}
%J. S\'{a}ndor, D. S. Mitrinovi\'{c}, B. Crstici, \textit{Handbook of Number Theory I}, Springer, 1995.

\bibitem{HKW}
J. G. Huard, P. Kaplan, K. S. Williams, The Chowla-Selberg formula for genera, \textit{Acta Arithmetica}, LXXIII.3 (1995).
\bibitem{HW}
G. H. Hardy, E. M. Wright, \textit{An Introduction to the Theory of Numbers}, sixth edition, ISBN 978-7-115-21427-0.
%\bibitem{Sol}
%J. Solymosi, Bounding multiplicative energy by the sumset, \textit{Advances in Mathematics} \textbf{222} (2009) 402-408.
\bibitem{Karatsuba}
A. A. Karatsuba, \textit{Basic Analytic Number Theory}, Springer-Verlag Berlin Heidelberg New York in 1993.
\bibitem{KN}
M. K\"{u}hleitner, W. Nowak, The average number of solutions of the Diophantine equation $U^2+V^2=W^3$ and related arithmetic functions, \textit{Acta Mathematica Hungarica} (2004), Vol. \textbf{104}: Issue 3, 225-240.
\bibitem{Kuhn}
S. K\"{u}hnlein, Partial solution of a conjecture of Schmutz, \textit{Arch. Math.}, Vol. \textbf{67}, 164-172 (1996).
\bibitem{MO}
P. Moree, R. Osburn, Two-dimensional lattices with few distances, \textit{L'Enseignement Math\'{e}matique} (2) \textbf{52} (2006), 361-380.
\bibitem{Muller}
W. M\"{u}ller, The mean square of Dirichlet series associated with automorphic forms, \textit{Monatsh. Math.} \textbf{113} (1992), 121-159.
\bibitem{CS}
A. Selberg, S. Chowla, On Epstein's zeta function, \textit{J. Reine Angew. Math.}, \textbf{227} (1967), 86-110.
\bibitem{SW}
Zhi-Hong Sun, K. S. Williams, On the number of representations of $n$ by $ax^2+bxy+cy^2$, \textit{Acta Arithmetica} 122.2 (2006).
\bibitem{Wilson}
B. M. Wilson, Proofs of some formulae enunciated by Ramanujan, \textit{Proc. London Math. Soc.} \textbf{21} (1922), 235-255.
\bibitem{ZW}
Zhang Nan-Yue, K. S. Williams, On the Epstein zeta function, \textit{Tamkang Journal of Mathematics}, Vol. \textbf{26}, No. 2, Summer 1995.


\end{thebibliography}
\end{document}